\documentclass[reqno]{amsart}
\usepackage{amsmath,amsthm,amsfonts,amssymb,amscd,enumerate,verbatim}
\usepackage{varwidth}
\usepackage{graphicx}
\usepackage[hidelinks=true]{hyperref}
\usepackage{colonequals}

\usepackage{tikz-cd}
\usepackage{stmaryrd}

\usepackage[all]{xy}\SelectTips{eu}{}

\usepackage{color}

\newcommand{\ann}{\operatorname{ann}}

\newcommand{\fn}{{\mathfrak n}}

\newcommand{\Ker}{\operatorname{Ker}}

\newcommand{\rank}{\operatorname{rank}}

\usepackage[capitalise, noabbrev]{cleveref}

\numberwithin{equation}{section}

\theoremstyle{plain}
\newtheorem{theorem}{Theorem}[section]

\newtheorem*{Theorem}{Theorem}

\newtheorem*{Main Theorem}{Main Theorem}

\newtheorem{lemma}[theorem]{Lemma}
\newtheorem{corollary}[theorem]{Corollary}

{\Alph{theorem}}
{\Alph{theorem}}

\theoremstyle{definition}

\newtheorem{construction}[theorem]{Construction}
\newtheorem{remark}[theorem]{Remark}
\newtheorem{definition}[theorem]{Definition}

  \newcounter{numlist} %
  {\end{list}}%

\theoremstyle{remark}

\newtheorem{chunk}[theorem]{}
\newenvironment{bfchunk}{\begin{chunk}\textbf}{\end{chunk}}

\numberwithin{equation}{theorem}

\numberwithin{equation}{theorem}

\newtheorem*{Claim}{Claim}
\newtheorem*{ind1}{Induction statement I}
\newtheorem*{ind2}{Induction statement II}

\newcommand\restr[2]{{
  \left.\kern-\nulldelimiterspace 
  #1 
  \vphantom{\big|} 
  \right|_{#2} 
  }}

\author[L.~M.~\c{S}ega]{Liana M.~\c{S}ega}
\address{Liana M.~\c{S}ega\\ School of Science and Engineering\\
   University of Missouri - Kansas City\\ 
   Kansas City\\ MO 64110\\ U.S.A.}
     \email{segal@umkc.edu}
     
     \author[D.~Sireeshan]{Deepak Sireeshan}
\address{Deepak Sireeshan\\ School of Science and Engineering\\
   University of Missouri - Kansas City\\ 
   Kansas City\\ MO 64110\\ U.S.A.}
     \email{dsbx7@mail.umkc.edu}

     \thanks{This work was supported in
part by a Simons Foundation grant (\#354594, Liana \c Sega)}

\begin{document}
\title[]{dg module structures and minimal free resolutions\\ modulo an exact zero divisor}
\begin{abstract}Let $Q$ be a local ring with maximal ideal $\mathfrak{n}$ and  let $f,g\in \mathfrak{n}\smallsetminus\mathfrak{n}^2$ with $fg=0$. When $M$ is a finite $Q$-module with $fM=0$, we show that a minimal free  resolution of $M$ over  $Q$ has a differential graded module structure over the differential graded algebra $Q\langle y,t\mid \partial(y)=f, \partial(t)=gy\rangle$. When $(f,g)$ is a pair of exact zero divisors, we use this structure to describe a minimal free resolution of $M$ over $Q/(f)$. 
\end{abstract}

\maketitle

\section*{introduction}

Differential graded (dg) structures  provide an effective and elegant tool in commutative algebra, with particularly strong applications in the study of homological behavior under a change of ring, see Avramov \cite{avramov} for the basics of the theory. 

Let  $\varphi\colon Q\to R$ be a ring homomorphism and $M$  an  $R$-module. One would like to be able to describe a projective resolution of $M$ over $R$, given a projective resolution  $A$ of $R$ over $Q$ and a projective resolution $U$ of $M$ over $Q$.  Iyengar \cite{iyengar} provides such a construction when the resolutions $A$ and $U$ have appropriate differential graded (dg) structures and discusses the existence of such structures.  When $Q$ is local (meaning also commutative noetherian)  with maximal ideal $\fn$ and $\varphi$ is surjective,  we are particularly interested in minimal free resolutions. In this case, if $A$ is a resolution of $R$ over $Q$ admitting a dg algebra structure and $U$ is a minimal free resolution of $M$ over $Q$ admitting a semi-free dg $A$-module structure, it is known that $U\otimes_A R$ is a minimal free resolution of $M$ over $R$, see ~\cref{l:standard}.  

When $f\in \fn\smallsetminus \fn^2$ and $R=Q/(f)$, a construction of Shamash can be used to build a semi-free dg module structure of $U$ over the Koszul complex on $f$, see \cite[Proposition 2.2.2]{avramov}. Consequently, this leads  to a description of a minimal free resolution of $M$ over $R$ when $f$ is, in addition, a non-zero divisor, cf.~\cite[Theorem 2.2.3]{avramov}.  In this paper, we extend the Shamash construction when  $f$ is a zero divisor. 

\begin{Theorem}
Let $(Q,\fn)$ be a local ring  and let $f,g\in \fn\smallsetminus \fn^2$ such that $fg=0$.  Set $R=Q/(f)$ and $A=Q\langle y,t\mid \partial(y)=f, \partial(t)=gy\rangle$. If $M$ is a finitely generated  $R$-module and $U$ is a minimal free resolution of $M$ over $Q$, then:
\begin{enumerate}[\quad\rm(a)]
\item  $U$ has a structure of semi-free dg module over $A$. 

\item Furthermore, if $\ann(f)=(g)$ and $\ann(g)=(f)$, then $U\otimes_AR$ is a minimal free resolution of $M$ over $R$.
\end{enumerate}
\end{Theorem}

The dg algebra  $A$  in the theorem is obtained from $Q$ by adjoining an exterior algebra variable $y$ in degree $1$  and a divided powers variable $t$ in degree $2$, see \cref{tate} for details on the notation, which is based on a construction usually referred to as a {\it Tate construction}. When the hypothesis of part (b) holds, we say that $(f,g)$ is a {\it pair of exact zero divisors} and that $f$, respectively $g$, is an {\it exact zero divisor}.  In this case, $A$ is a minimal free resolution of $R$ over $Q$. Although not all rings admit exact zero divisors, there exist significant classes of rings that do. See, for example,  Henriques and \c{S}ega \cite{zerodiv} for a treatment of a class of rings that are known to admit exact zero divisors, due to a result of Conca, Rossi, and Valla. Such elements can play an important role in understanding the structure and homological properties of artinian algebras. 

Let  $Q\to R$ be a surjective homomorphism of local rings such that a minimal free resolution $A$ of $R$ over $Q$ admits a dg algebra structure and $M$ is a finite $R$-module. The question of whether a minimal free resolution $U$ of $M$ over $Q$ admits a dg $A$-module structure was originally asked by Buchsbaum and Eisenbud \cite{be}, in the case when $A$ is a Koszul complex. Subsequently, this question was studied in the work of several authors, including Avramov,  Kustin, Iyengar, Miller, and Srinivasan, see \cite{obstructions}, \cite{iyengar}, \cite{kustin}, \cite{kustin-miller}, \cite{srinivasan1}, \cite{srinivasan2}, with both positive and negative answers. In all previously known results in which existence is established, either $A$ is a Koszul complex, or the resolution $U$ has a finite length.  Our result provides a first positive answer  (to our knowledge) in a case when both $A$ and $U$ are infinite. 

The minimal free resolution of $M$ over $R$  constructed in the theorem can be used to recover a known relationship \cite[Theorem 1.7]{zerodiv} between the betti numbers of $M$ over $Q$ and the betti numbers of $M$ over $R$ when $(f,g)$ is a pair of exact zero divisor, see \cref{r:betti}.  We note, however, that part (b) of the theorem, together with the construction of the dg module structure in the proof of (a) and the more detailed description of the complex $U\otimes_AR$ in Construction \ref{constr},  give a full description of the differential, and not just the betti numbers.

\section{The Tate construction}
In this section, we describe in detail the construction of the dg algebra $A$ in  the main theorem, and we set up notation that will be  used in the proof of the theorem. 

\begin{bfchunk}{Dg algebras and dg modules.}
Let $Q$ denote a commutative ring. If $W$ is a complex of $Q$-modules, we denote by $W^\#$ the underlying graded $Q$-module and we use $|w|$ for the degree of a homogeneous element $w$ of $W$. 

A {\it dg algebra} over $Q$ is a complex $A=(A, \partial)$, concentrated in non-negative homological degrees, where $A^\#$ is a $Q$-algebra that is (graded) commutative, meaning that 
$$
ab=(-1)^{|a||b|}ba\quad\text{for $a,b\in A$}\quad\text{and}\quad \text{$a^2=0$ when $|a|$ is odd}\,,
$$
and such that the Leibniz formula holds:  
$$
\partial(ab)=(\partial a)b+(-1)^{|a|}a(\partial b)\quad\text{for $a, b\in A$}\,.
$$
If $A$ and $B$ are dg $Q$-algebras, then $A\otimes_QB$ is also a dg-algebra, with multiplication given by 
\begin{equation}
\label{e:product}
(a\otimes b)(a'\otimes b')=(-1)^{|a'||b|}aa'\otimes bb'
\end{equation}
and the usual differential defined on a tensor product of complexes. 
We usually identify $A$ and $B$ with their images in $A\otimes_QB$ and we write $ab$ instead of $a\otimes b$. 

A {\it dg module} $U$ over $A$ is a complex $(U, \partial)$ with $U^\#$ an $A^\#$-module such that 
$$
\partial(au)=(\partial a)u+(-1)^{|a|}a(\partial u)\quad\text{for $a\in A$ and $u\in U$.}
$$
A bounded below dg $A$-module $U$ is said to be {\it semi-free} if $U^\#$ is free over $A^\#$. 
\end{bfchunk}

Existence of semi-free dg module structures is particularly useful, as they give an easy construction of free resolutions over $R$ from those over $Q$. For a more general construction, that does not require the semi-free assumption, see \cite{iyengar}. The benefit of the setting  in Lemma \ref{l:standard} below is that, in the local case, one obtains a minimal resolution when starting with a minimal one.

\begin{construction}\label{constr}
Let $Q\to R$ be a surjective homomorphism of commutative rings. Let $A$ be a dg $Q$-algebra algebra with $H_0(A)=R$ and let $U=(U, \partial)$ be a dg semi-free $A$-module.  Let  $\{e_{\lambda}\}_{\lambda\in \Lambda}$ denote a basis of  $U^\#$ over $A^\#$. For each integer $n$,  let $V_n$ denote  the free $Q$-summand of $U_n$ with basis $\{e_{\lambda}\mid \lambda\in \Lambda, |e_\lambda|=n\}$ and let $\pi_{n-1}\colon U_{n-1}\to V_{n-1}$ denote  the projection map.  Then $U\otimes_AR$ can be identified with the complex $U'=(U', \partial')$, with $$U'_n=V_n\otimes_QR\quad\text{and}\quad \partial'_n=\pi_{n-1}\restr{\partial_n}{V_n}\otimes_Q1_R\,.$$ 

Indeed, assume $R=Q/I$ for an ideal $I$, and $Q\to R$ is the natural projection. Let $\varepsilon\colon  A\to R$ denote the augmentation map and set  $J=\text{Ker}(\varepsilon)$. We have then $R=A/J$, and thus $U\otimes_AR\cong U/JU$. Let $n\ge 0$. The definition of $V_n$ gives $U_n=V_n\oplus W_n$, where $W_n=(A_{\ge 1}U)_n$. We have thus: 
$$
(U\otimes_AR)_n\cong \left(\frac{U}{JU}\right)_n=\frac{U_n}{IU_n+W_n}\cong \frac{U_n}{W_n}\otimes_QR\cong V_n\otimes_QR\,.
$$
The differential of $U'$ is induced by the one on $U\otimes_AR$, which is in turn induced by the differential of  $U$, and this yields directly the claimed expression for $\partial'$. 
\end{construction}

 The following is a standard argument that appears, for instance,  in the proof of \cite[Theorem 2.2.3]{avramov}. 

\begin{lemma}\label{l:standard}
Let $Q\to R$ be a surjective homomorphism of commutative rings. Let $A$ be a free  resolution of $R$ over $Q$ that has a structure of dg algebra, and let $U=(U, \partial)$ be a free resolution of $M$ over $Q$, that has a structure of semi-free dg-module over $A$. Then $U\otimes_AR$ is a free resolution of $M$ over $R$. 
In particular, if $Q$ is local and $U$ is minimal, then $U\otimes_AR$ is a minimal free resolution of $M$ over $R$.  
\end{lemma}

\begin{proof}
Since $A$ is a free resolution of $R$ over $Q$, the augmentation map $\varepsilon: A \to R$ is a quasi-isomorphism (that is, it induces an isomorphism in homology).  By \cite[Proposition 1.3.2]{avramov}, $\varepsilon\colon  A\to R$ induces a quasi-isomorphism $U\to U\otimes _AR$. This shows  that $U\otimes_AR$ is a free resolution of $M$ over $R$.  The last statement, regarding minimality, follows directly from Construction  \ref{constr}. 
\end{proof}

\begin{bfchunk}{The Tate construction.}\label{tate}
We describe below a basic procedure of adjoining variables to create new dg algebras, originally used by Cartan and further adapted by Tate. We refer to \cite{avramov} for the general construction, and we adapt it here to the case of interest. 

We denote by $Q\langle y\rangle$ the graded free $Q$-module with basis $1$, $y$, where $|y|=1$, and define a $Q$-linear multiplication by setting 
$$y^2=0\,.$$
Let $f\in Q$. We write $K=Q\langle y\mid \partial(y)=f\rangle$ for the dg $Q$-algebra  with underlying graded algebra $Q\langle y\rangle$ and differential satisfying $$\partial(y)=f\,.$$
 This is precisely a Koszul complex on $f$.

We denote by $Q\langle t\rangle$ the graded free module with basis $\{t^{(i)}\}$, where $|t^{(i)}|=2i$, and define a $Q$-linear multiplication by 
$$
t^{(0)}=1\quad \text{and}\quad  t^{(i)}t^{(j)}=\frac{(i+j)!}{i!j!}t^{(i+j)}\,.
$$ 
We set $t=t^{(1)}$; this element is called a {\it divided powers variable} and the basis elements $\{t^{(i)}\}$ are  called a {\it system of divided powers} of $t$. 

Let $g\in Q$ with $fg=0$, so that $gy$ is a cycle in $K$. 
We write $$A=Q\langle y,t\mid \partial(y)=f, \partial(t)=gy\rangle$$ for the dg $Q$-algebra with underlying graded algebra $K\otimes_QQ\langle t\rangle$, with differential extending the one on $K$ and satisfying 
$$
\partial(t^{(i)})=gyt^{(i-1)}\,.
$$
Note that we are using product notation for tensor products of elements, e.g. we are writing  $yt^{(n)}$ instead of $y\otimes t^{(n)}$. 
The multiplication on $A$ comes from the multiplication on $K$ and that on $Q\langle t\rangle$ using formula \eqref{e:product}. In particular, we have 
$$
yt^{(n)}=t^{(n)}y \quad\text{for all $n\ge 0$}\,.
$$
It is known and easy to check that, with these definitions, $A$ is indeed a dg algebra, that is referred to as the {\it Tate construction} on $f,g$. 
\end{bfchunk}

\begin{bfchunk}{Notation.}
We set  $$f_i=\begin{cases}
f &\text{if $i$ is odd}\\
g&\text{if $i$ is even.}\end{cases}$$
For all $i\ge 0$ we set
$$
y_{2i}=t^{(i)}\quad \text{and}\quad  y_{2i+1}=t^{(i)}y\,.
$$
We set $y_i=0$ if $i<0$. 
With this notation, note that $A$ is a free $Q$-module with basis $\{y_n\}_{n\ge 0}$ such that $|y_n|=n$ for all $n$, 
 and 
$$\partial_i(y_i)=f_iy_{i-1} \quad \text{for all $i$.}$$

The multiplicative structure of $A$ is given by the unit $y_0$ and the relations
\begin{align}
\label{y}
\begin{split}
y_{2i+1}y_{2j+1}&=0\\
y_{2i}y_{2j}&=\binom{i+j}{i}y_{2(i+j)}\\
y_{2i}y_{2j+1}&=\binom{i+j}{i}y_{2(i+j)+1} = y_{2j+1}y_{2i} 
\end{split}
\end{align}
for all $i,j\ge 0$. Using these relations, define $\alpha_{i,j}\in Q$ such that 
$$y_iy_j=\alpha_{i,j}y_{i+j}\quad\text{for all $i\ge 0$, $j\ge 0$}\,.$$

More precisely, we have $\alpha_{i,j}=0$ when both $i$ and $j$ are odd, $\alpha_{i,j}=\binom{\frac{1}{2}(i+j)}{\frac{i}{2}}$ if both $i,j$ are even, $\alpha_{i,j}=\binom{\frac{1}{2}(i+j-1)}{\frac{i}{2}}$ if $i$ is even and $j$ is odd, and $\alpha_{i,j}=\binom{\frac{1}{2}(i+j-1)}{\frac{i-1}{2}}$ if $i$ is odd and $j$ is even. 

For all $i,j\ge 0$, note that 
$$\alpha_{i,j}=\alpha_{j,i}\,.$$
Further, one can easily verify that 
\begin{align}
\label{2}
\alpha_{i,j}\alpha_{i+j,k}&=\alpha_{j,k}\alpha_{i,j+k}\qquad \text{for all $i,j,k\ge 0$}\\
\label{1}
\alpha_{i-1,j}f_i+(-1)^{i}\alpha_{i,j-1}f_j&=\alpha_{i,j}f_{i+j}\qquad\quad \text{for all $i,j\ge 1$\,.}
\end{align}
In fact,  \eqref{2} is equivalent to the associativity of $A$, namely  $(y_iy_j)y_k = y_i(y_jy_k)$, and \eqref{1} is equivalent to the Leibniz rule for $A$, namely: 
$$\partial(y_iy_j)=\partial(y_i)y_j+(-1)^iy_i\partial(y_j)\,.
$$
Finally, observe that the Tate construction $A$ is a periodic complex of rank $1$ free $Q$-modules: 
$$
\dots\to Qy_{2n+1}\xrightarrow{f}Qy_{2n}\xrightarrow{g}Qy_{2n-1}\xrightarrow{f}Qy_{2n-2}\to\dots\to Qy_2\xrightarrow{g}  Qy_1\xrightarrow{f}Qy_0\to 0\,,
$$ 
\end{bfchunk}

\begin{remark}
Let $A$ be as in \ref{tate}, and  set $R=Q/(f)$. Then, with the assumptions and notation introduced in Construction \ref{constr}, we have  an isomorphism of free $Q$-modules
\begin{equation}\label{sum}
U_n=V_n\oplus y_1V_{n-1}\oplus y_2V_{n-2}\oplus \cdots \oplus y_nV_0\cong \bigoplus_{i=0}^n V_n\,.
\end{equation}
\end{remark}

\begin{definition}[{\bf Exact zero divisors} ]
\label{ezd}
Let $f,g\in Q$. We say that $(f,g)$ is a {\it pair of exact zero divisors} if $\ann(f)=(g)$ and $\ann(g)=(f)$; we also refer to $f$ or $g$ as an {\it exact zero divisor}.  In this case, the Tate construction on $f,g$ is a minimal free resolution of $Q/(f)$ over $Q$. 
\end{definition}
\section{Proof of the main theorem}
Our main theorem is an extension of a construction of  Shamash \cite{Shamash}, along the lines of the proof in  \cite[Proposition 2.2.2, Theorem 2.2.3]{avramov}. 
We recall below the statement of the main theorem from the introduction and then proceed to prove it. A corollary  about betti numbers is given at the end of the section. 

\begin{theorem}\label{t:main}
Let $(Q,\fn)$ be a local ring  and let $f,g\in \fn\smallsetminus \fn^2$ such that $fg=0$.  Set $R=Q/(f)$ and $A=Q\langle y,t\mid \partial(y)=f, \partial(t)=gy\rangle$. If $M$ is a finitely generated  $R$-module and $U$ is a minimal free resolution of $M$ over $Q$, then:
\begin{enumerate}[\quad\rm(a)]
\item  $U$ has a structure of semi-free dg module over $A$. 

\item Furthermore, if $\ann(f)=(g)$ and $\ann(g)=(f)$, then $U\otimes_AR$ is a minimal free resolution of $M$ over $R$. 
\end{enumerate}
\end{theorem}

\begin{proof}
(a) For all integers $i\ge 0$ and $n$ we will construct homomorphisms 
$$\sigma_{i,n}\colon U_n\to U_{i+n}$$
and free $Q$-modules  $V_n$ with $V_{n}\subseteq U_{n}$ such that: 
\begin{enumerate}[(1)]
\item $\sigma_{0,n}=\text{id}_{U_n}$ for all $n$. 
\item For all $n$ and $i\ge 1$ the following condition holds:
\begin{equation}
\label{A}
\mathcal A(i,n): \qquad \partial_{i+n}\sigma_{i,n}+(-1)^{i+1}\sigma_{i,n-1}\partial_{n}=f_i\sigma_{i-1,n}\,.
\end{equation}
\item For all $n$ and all $i,j$ with $i\ge 0$ and $j\ge 0$ the following condition holds: 
$$
\mathcal B(i,j,n):\qquad\sigma_{i, n}\sigma_{j,n-j}=\alpha_{i,j}\sigma_{i+j,n-j}\,.
$$

\item  For all $n$ the following condition holds: 
\begin{align*}
\mathcal C(n): \quad&\text{The map}\quad 
\Phi_n\colon V_0\oplus V_1\oplus \dots \oplus V_{n}\to U_{n+1}\quad\text{given by}\\
&\qquad \Phi_n(x_0,x_1, \dots, x_{n})=\sum_{i=1}^{n+1} \sigma_{i,n+1-i}(x_{n+1-i})\\
&\text{is a split injection.}
\end{align*}

\item For all $n$ the following holds: 
$$\mathcal D(n):\quad U_n=\sum_{k=0}^n\sigma_{k,n-k}(U_{n-k})=\bigoplus_{k=0}^n\sigma_{k,n-k}(V_{n-k})\,,
$$
where the direct sum indicates an  {\it internal} direct sum, that is, every element of $U_n$ can be written uniquely as a sum of elements in $\sigma_{k,n-k}(V_{n-k})$. 
\end{enumerate}

Once these conditions are established,  we define 
$$y_iu=\sigma_{i,n}(u)\quad\text{for all $n$, $i\ge 0$ and $u\in U_n$}$$  
and then extend linearly to a multiplication $A\times U\to U$. Conditions (1)-(3)  imply that this rule defines a dg $A$-module structure on $U$.  Item (5) shows that $U$ is a semi-free dg $A$-module. Indeed, condition $\mathcal D(n)$ can be re-written as 
$U_n=\bigoplus_{k=0}^ny_kV_{n-k}$, implying that 
a basis of $U^\#$ over $A^\#$ is $\cup_{i\ge 0}B_i$, where $B_i$ denotes a $Q$-basis of $V_i$. 

Before proceeding with the bulk of the proof, we show: 

\begin{Claim}
Let $n$ be an integer. Assume $\mathcal C(m)$, $\mathcal D(m)$ and $\mathcal B(i,j,m)$ hold for all $m$ with $0\le m\le n-1$ and all $i,j\ge 0$. Then 
\begin{equation}
\label{claim}
\sum_{k=j}^n\sigma_{k,n-k}(U_{n-k})=\bigoplus_{k=j}^n \,\sigma_{k,n-k}(V_{n-k})\subseteq U_n\quad\text{for all $j\ge 1$}\,.
\end{equation}
\end{Claim}

\noindent{\it Proof of Claim.}
Let $j\ge 1$. We have: 
\begin{align*}
\sum_{k=j}^n\sigma_{k,n-k}(U_{n-k})&=\sum_{k=j}^n \sum_{k'=0}^{n-k} \sigma_{k,n-k}\sigma_{k',n-k-k'}(V_{n-k-k'})\\
&\subseteq \sum_{k=j}^n \sum_{k'=0}^{n-k}\sigma_{k+k', n-(k+k')}(V_{n-(k+k')})\\&
\subseteq \sum_{l=j}^n \sigma_{l,n-l}(V_{n-l})\\
&\subseteq \sum_{k=j}^n\sigma_{k,n-k}(U_{n-k})\,.
\end{align*}
We used $\mathcal D(n-k)$ in the first line, $\mathcal B(k,k', n-k)$ in the second, and $V_i\subseteq U_i$ in the fourth. (Note that $k\ge j\ge 1$ and thus $n-k\le n-1$.) It follows that all inclusions above are in fact equalities. Further, the sum $\sum_{k=j}^n \sigma_{k,n-k}(V_{n-k})$ is an internal direct sum, as can be seen from $\mathcal C(n-1)$. This finishes the proof of the claim. 
\medskip

We now proceed to construct the maps $\sigma_{i,n}$ and the modules $V_n$.
We start by setting $\sigma_{*,n}=0$ and $V_n=0$ whenever $n<0$. With this definition, note  that (1)-(5) hold when $n<0$. We proceed by induction.
 
 Consider the following statement, depending on an integer $k$. 

\begin{ind1} The maps $\sigma_{i,k}\colon U_k\to U_{i+k}$ are defined for all $i\ge 0$ and  the free $Q$-module $V_k$ with $V_k\subseteq U_k$ are defined, and satisfy the properties: 
\begin{enumerate}[\qquad $\bullet$]
\item $\sigma_{0,k}=\text{id}_{U_k}$
\item $\mathcal A(i,k)$ holds for all $i\ge 1$
\item  $\mathcal B(i,j,k)$ holds for all $i\ge 0$, $j\ge 0$
\item $\mathcal C(k)$ and $\mathcal D(k)$  hold. 
\end{enumerate}
\end{ind1}

As noted above, this statement holds when $k<0$. Let $n\ge 0$. Assume that Induction statement I holds for all $k\le n-1$.  
To complete the induction, we define next $\sigma_{*,n}$ and  $V_n$, and we show that all four items in Induction statement I hold when $k=n$. 
\smallskip

We start by defining $\sigma_{0,n}$ to be the identity map on $U_n$. Then,  we define the free module $V_n$ as follows: Condition $\mathcal C(n-1)$ gives that $\text{Im}\Phi_{n-1}$ is a direct summand of $U_n$. We define $V_n$ to be the complementary direct summand, so that 
$$
U_n=\text{Im}\Phi_{n-1}\oplus V_n=\text{Im}\Phi_{n-1}\oplus \sigma_{0,n}(V_n)\,.
$$
Observe that $\mathcal D(n)$ must then hold, as it is a direct consequence of this definition and of $\mathcal C(n-1)$. We need to define now $\sigma_{i,n}$ for $i\ge 1$. We proceed by induction. 

With $n$ fixed as above, consider the following statement, depending on an integer $l\ge 0$:

\begin{ind2}   The map $\sigma_{l,n}\colon U_n\to U_{l+n}$ is defined such that 
\begin{enumerate}[\qquad $\bullet$]
\item $\mathcal A(l,n)$ holds
\item  $\mathcal B(l,j,n)$ holds for all $j\ge 0$.
  \end{enumerate}
  \end{ind2}
  
 Note that these two conditions hold trivially when $l=0$.  Let $i\ge 1$ and assume that Induction statement II holds for all $l$ with $l\le i-1$. We now define $\sigma_{i,n}$ and we show that the two  items in Induction Statement II also hold with $l=i$. 
 
  In view of the  direct sum decomposition of $U_n$ given by condition $\mathcal D(n)$ (which we know holds, see above), in order to define $\sigma_{i,n}$  it suffices to define the restriction functions $\restr{\sigma_{i,n}}{\sigma_{k,n-k}(V_{n-k})}$ for all $k$ with $0\le k\le n$. 

Assume first $k>0$. Note that $\mathcal C(n-k)$ implies that the restricted function 
$$\restr{\sigma_{k,n-k}}{V_{n-k}}\colon V_{n-k}\to \sigma_{k,n-k}(V_{n-k})$$
 is bijective. We define then
 $$
 \restr{\sigma_{i,n}}{\sigma_{k,n-k}(V_{n-k})}:=\alpha_{i,k}\restr{\sigma_{i+k,n-k}}{V_{n-k}}(\restr{\sigma_{k,n-k}}{V_{n-k}})^{-1}\,.
 $$
 In other words, this definition ensures that 
 \begin{equation}
 \label{e:restricted}
\restr{\sigma_{i,n}\sigma_{k,n-k}}{V_{n-k}}=\alpha_{i,k}\restr{\sigma_{i+k,n-k}}{V_{n-k}}\quad \text{when $k>0$}\,.
 \end{equation}
We need to define now $\restr{\sigma_{i,n}}{V_n}$. To do so, we first claim that 
\begin{equation}
\label{need}
\left(f_i\sigma_{i-1,n}-(-1)^{i+1}\sigma_{i,n-1}\partial_n\right)(U_n)\subseteq \partial_{i+n}(U_{i+n})\,.
\end{equation}
Assuming \eqref{need} holds and recalling that $V_n$ is a free $Q$-module,  we can define $\restr{\sigma_{i,n}}{V_n}\colon V_n\to U_{i+n}$ so that 
$$
\partial_{i+n}\restr{\sigma_{i,n}}{V_n}=\restr{f_i\sigma_{i-1,n}-(-1)^{i+1}\sigma_{i,n-1}\partial_n}{V_n}\,.
$$
Indeed, fix a basis of $V_n$ and define $\restr{\sigma_{i,n}}{V_n}(e)$ for a basis element $e$ to be the preimage of  $\left(f_i\sigma_{i-1,n}-(-1)^{i+1}\sigma_{i,n-1}\partial_n\right)(e)$ under $\partial_{i+n}$, and then extend by linearity. Note that this  definition depends on the choice of preimage, and thus is not unique. 

To prove \eqref{need}, we compute: 
\begin{align*}
\partial_{i+n-1}(f_i\sigma_{i-1,n}&-(-1)^{i+1}\sigma_{i,n-1}\partial_n)\\
&=f_i\partial_{i+n-1}\sigma_{i-1,n}+(-1)^i\partial_{i+n-1} \sigma_{i,n-1}\partial_n\\
&=f_i\left(f_{i-1}\sigma_{i-2,n}-(-1)^i \sigma_{i-1,n-1}\partial_n\right)+\\
&\qquad\qquad +(-1)^i\left(f_i\sigma_{i-1,n-1}\partial_n-(-1)^{i+1}\sigma_{i-1,n-2}\partial_{n-1}\partial_n\right)\\
&=0\,,
\end{align*}
where for the second equality we used the induction hypothesis that $\mathcal A(i-1,n)$ and $\mathcal A(i, n-1)$ hold, and for the third equality we used $f_if_{i-1}=fg=0$ and $\partial^2=0$.  The computation above shows
$$
\left(f_i\sigma_{i-1,n}-(-1)^{i+1}\sigma_{i,n-1}\partial_n\right)(U_n)\subseteq\Ker(\partial_{i+n-1})\,.
$$
Since $i\ge 1$ and $n\ge 0$, we have $i+n-1\ge 0$. When $i+n-1>0$, we use the fact that $U$ is acyclic, and hence $\text{Ker}(\partial_{i+n-1})=\text{Im}(\partial_{i+n})$ and thus \eqref{need} holds. 
When $i=1$ and $n=0$ we have 
$$
f_1\sigma_{0,0}-(-1)^{2}\sigma_{1,-1}\partial_0=f_1\text{id}_{U_0}\,.
$$
Since $U$ is a minimal free resolution of $M$, we have $M=U_0/\partial_1(U_1)$. Since $fM=0$, we have  $f_1U_0\subseteq \partial_1(U_1)$. This finishes the proof of \eqref{need} and the definition $\sigma_{i,n}$. 

Let $j\ge 0$. We want to prove that $\mathcal B(i,j,n)$ holds. Since $\sigma_{0,n}$ is the identity map, $\mathcal B(i,0,n)$ holds. 

 Assume now $j>0$. In view of the direct sum decomposition provided by $\mathcal D(n-j)$,  in order to prove $\mathcal B(i,j,n)$, it suffices to show
$$\restr{\sigma_{i,n}\sigma_{j,n-j}}{\sigma_{k,n-j-k}(V_{n-j-k})}=\alpha_{i,j}\restr{\sigma_{i+j,n-j}}{\sigma_{k,n-j-k}(V_{n-j-k})}
$$
for all $k$ with $0\le k\le n-j$. We need to check that
$$\restr{\sigma_{i,n}\sigma_{j,n-j}\sigma_{k,n-j-k}}{V_{n-j-k}}=\alpha_{i,j}\restr{\sigma_{i+j,n-j}\sigma_{k,n-j-k}}{V_{n-j-k}}\,.
$$
Indeed, we have: 
\begin{align*}
\restr{\sigma_{i,n}\sigma_{j,n-j}\sigma_{k,n-j-k}}{V_{n-j-k}}
&=\alpha_{j,k}\restr{\sigma_{i,n}\sigma_{j+k, n-(j+k)}}{V_{n-(j+k)}}\\
&=\alpha_{j,k}\alpha_{i,j+k}\restr{\sigma_{i+j+k, n-(j+k)}}{V_{n-(j+k)}}\\
&=\alpha_{i,j}\alpha_{i+j,k}\restr{\sigma_{i+j+k, n-(j+k)}}{V_{n-(j+k)}}\\
&=\alpha_{i,j}\restr{\sigma_{i+j,n-j}\sigma_{k,n-j-k}}{V_{n-j-k}}\,.
\end{align*}
Here, we used $\mathcal B(j,k,n-j)$ in the first equality, \eqref{e:restricted} in the second, \eqref{2} in the third,  and $\mathcal B(i+j,k,n-j)$ 
 in the last. This finishes the proof of $\mathcal B(i,j,n)$.

Next,  we want to prove that $\mathcal A(i,n)$ holds. Given the definition of $\restr{\sigma_{i,n}}{V_n}$, we know that the relation holds when we restrict the functions to $V_n$. We need to also check this relation when the functions are restricted to $\sigma_{k,n-k}(V_{n-k})$ with $k>0$. To this extent, it suffices to show 
$$
(f_i\sigma_{i-1,n}-(-1)^{i+1}\sigma_{i,n-1}\partial_n)\sigma_{k,n-k}=\partial_{i+n}\sigma_{i,n}\sigma_{k,n-k}\quad\text{for all $k>0$}\,.
$$
Let $k>0$. The computation below achieves the desired conclusion: 
\begin{align*}
(f_i&\sigma_{i-1,n}-(-1)^{i+1}\sigma_{i,n-1}\partial_n)\sigma_{k,n-k} \\
&=f_i\sigma_{i-1,n}\sigma_{k,n-k}-(-1)^{i+1}\sigma_{i,n-1}\partial_n\sigma_{k,n-k}\\
&=\alpha_{i-1,k}f_i\sigma_{i-1+k,n-k}-(-1)^{i+1}\sigma_{i,n-1}(f_k\sigma_{k-1, n-k}-(-1)^{k+1}\sigma_{k,n-k-1}\partial_{n-k})\\
&=\alpha_{i-1,k}f_i\sigma_{i-1+k,n-k}+(-1)^{i}\alpha_{i,k-1}f_k\sigma_{i+k-1,n-k}+(-1)^{i+k}\alpha_{i,k}\sigma_{i+k,n-k-1}\partial_{n-k}\\
&=\alpha_{i,k}f_{i+k}\sigma_{i+k-1,n-k}-\alpha_{i,k}(-1)^{i+k+1}\sigma_{i+k,n-k-1}\partial_{n-k}\\
&=\alpha_{i,k}\partial_{i+n}\sigma_{i+k,n-k}\\
&=\partial_{i+n}\sigma_{i,n}\sigma_{k,n-k}\,.
\end{align*}
We used $\mathcal B(i-1,k,n)$ and $\mathcal A(k,n-k)$ for the second equality, $\mathcal B(i,k-1,n-1)$ and $\mathcal B(i,k,n-1)$ for the third, \eqref{1} for the fourth, $\mathcal A(i+k, n-k)$ for the fifth and $\mathcal B(i,k,n)$ (which was proved earlier) for the last equality.

At this point, we verified Induction statement II for $l=i$, hence we have a definition of $\sigma_{i,n}$ for all $i$, such that $\mathcal A(i,n)$ holds and $\mathcal B(i,j,n)$ holds for all $i,j\ge 0$. We need to finalize the proof of the Induction statement I for $k=n$. The only remaining item is to show $\mathcal C(n)$, namely that the map $\Phi_n$ is a split injective. We will show that if $\Phi_n(x)\in \fn U_{n+1}$ for 
 $$x=(x_0, x_1, \dots, x_{n})\in V_0\oplus V_1\oplus \dots \oplus V_n\,,$$
 then $x_i\in \fn V_i$ for all $i$. 

Assume $\Phi_n(x)\in \fn U_{n+1}$. 
Then $\partial_{n+1}\Phi_n(x)\in \fn^2U_{n}$. For the purposes of the next computation, we set $x_{n+1}=0$. 
We have then:
\begin{align}
\begin{split}
\label{e:compute}
\partial&_{n+1}\Phi_n(x)=\sum_{i=1}^{n+1}\partial_{n+1}\sigma_{i,n+1-i}(x_{n+1-i})\\
&=\sum_{i=1}^{n+1}f_i\sigma_{i-1,n+1-i}(x_{n+1-i})-(-1)^{i+1}\sigma_{i,n-i}\partial_{n+1-i}(x_{n+1-i})\\
&=\sum_{j=0}^{n}f_{j+1}\sigma_{j,n-j}(x_{n-j})+\sum_{j=0}^{n}(-1)^{j}\sigma_{j,n-j}\partial_{n+1-j}(x_{n+1-j})\\
&=\sum_{j=0}^n\sigma_{j,n-j}\left(f_{j+1}x_{n-j}-\partial_{n-j+1}(x_{n-j+1})\right)\,.
\end{split}
\end{align}
We used $\mathcal A(i,n+1-i)$ for the second equality. Next, we set
$$
w_j=\sigma_{j,n-j}\left(f_{j+1}x_{n-j}-\partial_{n-j+1}(x_{n-j+1})\right)\,,
$$
so that \eqref{e:compute} becomes
$$
\partial_{n+1}\Phi_n(x)=\sum_{j=0}^n w_j\,.
$$
We prove by induction on $i$ with $0\le i\le n+1$  that $x_{n+1-i}\in \fn V_{n+1-i}$. This is true when $i=0$, as we defined $x_{n+1}=0$.  Let $k$ be so that $0 \leq k \leq n$ and assume that $x_{n+1-i}\in \fn V_{n+1-i}$ for all $i$ with $0\le i\le k$.  The fact that that $f_i\in \fn$ for all $i$ and the minimality of $U$ imply
$$w_j\in \fn^2V_{n-j}\quad\text{for all}\quad j\le k-1\,.$$
To complete the induction argument, we show that $x_{n-k}\in \fn V_{n-k}$. 
We have
\begin{align*}
f_{k+1}\sigma_{k,n-k}(x_{n-k})+\sum_{j=k+1}^n w_j&=w_k+\sigma_{k,n-k}\partial_{n-k+1}(x_{n-k+1})+\sum_{j=k+1}^n w_j\\
&=\sigma_{k,n-k}\partial_{n-k+1}(x_{n-k+1})+\partial_{n+1}\Phi_n(x)-\sum_{j=0}^{k-1} w_j\\
&\in \fn^2U_{n}\,.
\end{align*}
Recall the direct sum decomposition of $U_n$  given by property $\mathcal D(n)$, and notice that $f_{k+1}\sigma_{k,n-k}(x_{n-k})$ and $\sum_{j=k+1}^n w_j$ belong to distinct summands in this decomposition. Indeed, we have $f_{k+1}\sigma_{k,n-k}(x_{n-k})\in \sigma_{k,n-k}(V_{n-k})$, while  \eqref{claim} implies
$$
\sum_{j=k+1}^nw_j\in \sum_{j=k+1}^{n}\sigma_{j,n-j}(U_{n-j})=\bigoplus_{j=k+1}^{n}\sigma_{j,n-j}(V_{n-j})\,.
$$
We conclude $f_{k+1}\sigma_{k,n-k}(x_{n-k})\in \fn^2 U_n$, and hence $\sigma_{k,n-k}(x_{n-k})\in \fn U_n$, since $f_{k+1}\notin \fn^2$. Since $\Phi_{n-1}$ is a split injection by $\mathcal C(n-1)$,  it follows $x_{n-k}\in \fn V_{n-k}$.

We showed thus that $\mathcal C(n)$ holds.  This finishes the proof of Induction Statement I and concludes the proof of (a). 

(b) The fact that $(f,g)$ is a pair of exact zero divisors implies that $A$ is a minimal free resolution of $R$ over $Q$, as noted in \cref{ezd}.  The conclusion follows then from \cref{l:standard}, making use of (a). 
\end{proof}

\begin{remark}
The proof of the Main Theorem is inspired by the construction of Shamash described in \cite[Proposition 2.2]{avramov}. It is proved there that if $f\in \fn\smallsetminus \fn^2$, then, with notation as in our theorem, $U$ has a structure of semi-free dg module over the Koszul complex $K=Q\langle y\mid \partial(y)=f\rangle$. As noted in \cite[Remark 2.2.1]{avramov}, since $f\text{id}^U$ and $0^U$ both induce the zero map on $M$, they are homotopic, say $f\text{id}^U=\partial\sigma+\sigma\partial$, and the existence of the desired dg module structure is equivalent with the existence of a homotopy $\sigma$ with $\sigma^2=0$. The map $\sigma_{1,*}$ constructed in our proof above is exactly the homotopy $\sigma$ constructed in the proof of \cite[Proposition 2.2]{avramov}. In fact, all maps $\sigma_{i,*}$ with $i\ge 1$ are homotopies. Indeed, 
since $f_{i-1}f_i=0$, we see from condition $\mathcal A(i,n)$ in \eqref{A} that  $f_{i-1}\sigma_{i, *}$ is a chain map of degree $i$ on $U$, for all $i\ge 1$. Consequently, $\mathcal A(i,n)$ shows that $\sigma_{i,*}$ is a homotopy between 
$f_i\sigma_{i-1, *}$ and the zero map on $U$.  
\end{remark}
\begin{remark}
\label{r:betti}
Assume that the hypothesis of the Main Theorem holds. In the proof of the theorem, we saw that a basis for $U^\#$ over $A^\#$ can be taken to consist of the union of the bases  of the free $Q$-modules $V_0, V_1, V_2, \dots $. With this choice of basis, the modules $V_n$ defined in the proof of the theorem coincide with the modules $V_n$ introduced  in Construction \ref{constr}, and thus the differential of the complex $U\otimes_AR$ can be understood as described there.  Further, Construction  \ref{constr} and \eqref{sum} give
\begin{equation}\label{UV}
    U_n \cong \bigoplus_{i=0}^n V_i\qquad\text{and}\qquad (U\otimes_AR)_n\cong V_n\otimes_QR\,.
\end{equation}
\end{remark}

When $R$ is a local ring and $M$ is a finitely generated $R$-module, we let $\beta_n^R(M)$ denote the $n$th {\it Betti number} of $M$ over $R$, that is, the rank of the $n$th free $R$-module in a minimal free resolution of $M$, and we let 
$$P_M^R(t)=\sum_{i=0}^\infty \beta_i^R(M)t^i$$
denote the {\it Poincar\'e series} of $M$ over $R$.

 Corollary \ref{last} below recovers the Poincar\'e series formula in \cite[Theorem 1.7]{zerodiv}, and shows that the hypothesis on $f$, $g$ in part (b) of \cref{t:main} is necessary. 

\begin{corollary}
\label{last}
Let $(Q,\fn)$ be a local ring and $f,g\in \fn\smallsetminus \fn^2$ with $fg=0$. Set $R=Q/(f)$ and $A=Q\langle y,t\mid \partial(y)=f, \partial(t)=gy\rangle$. The following are equivalent: 
\begin{enumerate}[\quad\rm(1)]
\item $(f, g)$ is a pair of exact zero divisors;
\item For all finitely generated $R$-modules $M$, if $U$ is a minimal free resolution of $M$ over $Q$, then  $U\otimes_AR$ is a minimal free resolution of $M$ over $R$;
\item For all finitely generated $R$-modules $M$,  
$P_M^R(t)=(1-t)P_M^Q(t)$. 
\end{enumerate}
\end{corollary}

\begin{proof}
The implication (1)$\implies$(2) is established in \cref{t:main}(b). 

(2)$\implies$(3):  Let $n\ge 0$. Under the hypothesis of (2), we have $\beta_n^R(M)=\rank_R(U\otimes_AR)_n$ and $\beta_n^Q(M)=\rank_Q(U_n)$. Then \eqref{UV} gives 
$$
\beta^Q_n(M)=\sum_{i=0}^n \beta_i^R(M)
$$
and the Poincar\'e series formula follows from here. 

(3)$\implies$(1): If (3) holds, take $M=R$ in the Poincar\'e series formula to conclude $P^Q_R(t)=(1-t)^{-1}$, and hence $\beta_n^Q(R)=1$ for all $n\ge 0$. Since $\ann(f)$ is a second syzygy in a minimal free resolution of $R=Q/(f)$ over $Q$, we conclude that $\ann(f)$ is principal. If $\ann(f)=(h)$ for some $h\in \fn$, then, since $g\in (h)$ and $g\notin \fn^2$, we must have $g=uh$ for some unit $u$, and hence $\ann(f)=(h)=(g)$. Then $\ann(g)$ is a third syzygy in a minimal free resolution of $R$ over $Q$, and we conclude it must be principal as well. Similarly, we obtain $\ann(g)=(f)$.
\end{proof}

\subsubsection*{Acknowledgments} 
We are thankful to the referee for a careful reading, and useful suggestions.

\end{document}